\documentclass[12pt,reqno]{article}

\usepackage[usenames]{color}
\usepackage{amssymb}
\usepackage{amsmath}
\usepackage{amsthm}
\usepackage{amsfonts}
\usepackage{amscd}
\usepackage{graphicx}

\usepackage[colorlinks=true,
linkcolor=webgreen,
filecolor=webbrown,
citecolor=webgreen]{hyperref}

\definecolor{webgreen}{rgb}{0,.5,0}
\definecolor{webbrown}{rgb}{.6,0,0}

\usepackage{color}
\usepackage{fullpage}
\usepackage{float}

\usepackage{graphics}
\usepackage{latexsym}
\usepackage{epsf}
\usepackage{breakurl}

\newcommand{\seqnum}[1]{\href{https://oeis.org/#1}{\rm \underline{#1}}}
\def\Enn{\mathbb{N}}

\begin{document}

\theoremstyle{plain}
\newtheorem{theorem}{Theorem}
\newtheorem{corollary}[theorem]{Corollary}
\newtheorem{lemma}[theorem]{Lemma}
\newtheorem{proposition}[theorem]{Proposition}

\theoremstyle{definition}
\newtheorem{definition}[theorem]{Definition}
\newtheorem{example}[theorem]{Example}
\newtheorem{conjecture}[theorem]{Conjecture}

\theoremstyle{remark}
\newtheorem{remark}[theorem]{Remark}

\title{Cloitre's Self-Generating Sequence}

\author{Jeffrey Shallit\footnote{Research supported by a grant from NSERC,
2024-03725.}\\
School of Computer Science\\
University of Waterloo\\
Waterloo, ON N2L 3G1 \\
Canada\\
\href{mailto:shallit@uwaterloo.ca}{\tt shallit@uwaterloo.ca}}

\maketitle

\begin{abstract}
In 2009 Benoit Cloitre introduced a certain
self-generating sequence 
$$(a_n)_{n\geq 1} = 1, 1, 2, 1, 1, 1, 1, 2, 1, 1, 2, 1, 1, 2, 2, \ldots,$$
with the property that the sum of the terms appearing in the $n$'th run
equals twice the $n$'th term of the sequence.  We give a connection between
this sequence and the paperfolding sequence, and then prove Cloitre's
conjecture about the density of $1$'s appearing in $(a_n)_{n \geq 1}$.
\end{abstract}

\section{Introduction}

Recall that a {\it run\/} in a sequence is a maximal block of consecutive
identical values \cite{Golomb:1966}.
The starting point for this paper was a 2009 conjecture by Benoit Cloitre
that appears in the entry for sequence \seqnum{A157196} in the
On-Line Encyclopedia of Integer Sequences (OEIS) \cite{oeis}.  Paraphrased,
his conjecture reads as follows.
\begin{conjecture}
Let $(a_n)_{n \geq 1}$ be the unique sequence 
$$\underbrace{1,1,}_2 \, \underbrace{2,}_2 \, \underbrace{1,1,1,1,}_4 \,
\underbrace{2,}_2 \, \underbrace{1,1,}_2 \, \underbrace{2,}_2 \,
\underbrace{1,1,}_2 \,
\underbrace{2,2,}_4 \, \underbrace{1,1,}_2 \,
\underbrace{2,}_2 \, \underbrace{1,1,1,1,}_4 \ldots$$ 
over the alphabet 
$\{1,2 \}$, beginning with $a_1=1$, with the property that the sums
of the elements appearing in the $n$'th run equals twice
the sequence itself.  (The sums of the runs appear underneath the sequence.)

Then the number of $1$'s appearing
in the prefix $a_1 a_2 \cdots a_n$ is $2n/3 + o(n)$.
\end{conjecture}

For example, $a_1 = 1$, so the first run must have two elements, both
of which are $1$'s, so the sequence begins $1,1$.  The next term
then is forced to be a single $2$.  This forces the next four elements
to be $1$, and so forth.

Cloitre's sequence is reminiscent of another sequence, the celebrated
Oldenburger-Kolakoski sequence 
$$ {\bf k} := 1, 2, 2, 1, 1, 2, 1, 2, 2, 1, 2, 2, 1, 1, 2, 1, 1, 2, 2, 1, 2, 1, 1,\ldots$$
defined by
the property that the $n$'th term of the sequence is the length
of the $n$'th run.  See, for example,
\cite{Oldenburger:1939,Kolakoski:1965}.
The properties of $\bf k$ are still quite mysterious; in particular,
understanding the asymptotic behavior of
the number of $1$'s in a prefix
of length $n$ seems hard.  It seems reasonable to conjecture
that this number is $n/2 + o(n)$, but this is
not currently known.

In this paper we prove Cloitre's conjecture about the sequence 
$(a_n)_{n\geq 1}$.
This sequence is somehow much more tractable than the Oldenburger-Kolakoski 
sequence, because it is related
to the binary expansion of $n$ in a rather subtle way.
In this, Cloitre's sequence $(a_n)_{n \geq 1}$ resembles another Kolakoski-like
sequence more than the Oldenburger-Kolakoski sequence itself:  namely,
$(a_n)_{n \geq 1}$ resembles
the Dekking sequence $1,3,3,3,1,1,1,3,3,3,\ldots$ \cite{Dekking:1997},
in which we also have that the sequence of run lengths gives the
sequence itself.  Dekking's sequence is more tractable to understand, as it
is the image of a fixed point of a morphism.

Along the way we will encounter many interesting sequences, all of
which can be computed by finite automata, and these finite automata
can be turned into proofs of our assertions.

\subsection{Finite automata}

We assume the reader has some basic familiarity with finite automata,
as discussed, for example, in \cite{Hopcroft&Ullman:1979}.  
Each finite automaton has a number of states, depicted as circles,
and labeled transitions between them, depicted as arrows.  A circle
with a headless arrow entering denotes the (unique) initial state,
and a double circle denotes an accepting or final state.  
The first number inside a circle is the state number, and the second
number (if it is there) denotes the output associated with that state.

Recall that a state $q$ of an automaton is called {\it dead\/} if,
starting at $q$, no accepting state can be reached.
To save space, in this paper
some automata are displayed without the dead state, nor transitions
into or out of the dead state.

\section{Paperfolding sequences}

As it turns out, Cloitre's sequence is related to the regular paperfolding
sequence.  In this section we introduce this sequence as a member of
a larger class of similar sequences.

Paperfolding sequences are sequences over 
$\{ -1, 1\}$ arising from iterated folding of a piece of paper, where one
can introduce a hill ($+1$) or valley ($-1$) at each fold.  
See, for example, \cite{Davis&Knuth:1970,Dekking&MendesFrance&vanderPoorten:1982}.

Formally, let $\bf f$ be a finite sequence over $\{ -1, 1 \}$, and define
\begin{align}
P_\epsilon &= \epsilon \nonumber\\
P_{{\bf f} a} &= (P_{\bf f}) \  a \ ({-P_{{\bf f}}^R}) \label{fund}
\end{align}
for $a \in \{ -1, 1\}$ and ${\bf f} \in \{-1, 1\}^*$.
Notation:  $\epsilon$ denotes the empty sequence of length $0$,
$-x$ changes the sign of each element of a sequence $x$, and
$x^R$ reverses the order of symbols in a sequence $x$.

Now let ${\bf f} = f_0 f_1 f_2 \cdots$ be an infinite sequence in
$\{-1, 1\}^\omega$.  It is easy to see that 
$P_{f_0 f_1 \cdots f_n}$ is a prefix of
$P_{f_0 f_1 \cdots f_{n+1}}$ for all $n \geq 0$, so there is a unique
infinite sequence of which all the $P_{f_0 f_1 \cdots f_n}$
are prefixes; we call this infinite sequence $P_{\bf f}$.

We index the unfolding instructions starting
at $0$:  ${\bf f} = f_0 f_1 f_2 \cdots$, while the
paperfolding sequence itself is indexed starting at $1$:
$P_{\bf f} = p_1 p_2 p_3 \cdots$.
Then clearly $P_{\bf f} [2^n] = p_{2^n} = f_n$
for $n \geq 0$.  Hence there are
uncountably many infinite paperfolding sequences.

The most famous such sequence is the
{\it regular paperfolding sequence}, corresponding to the
sequence of unfolding instructions $1^\omega = 111\cdots$.  
For example
\begin{align*}
P_1 &= 1 \\
P_{11} &= 1 \, 1 \, (-1) \\
P_{111} &= 1 \, 1 \, (-1) \, 1 \, 1 \, (-1) \, (-1) .
\end{align*}
These are all prefixes of the infinite paperfolding sequence
$(p_n)_{n \geq 1} = 1,1,-1,1,1,-1,-1,\ldots$.  

In this paper we are mostly concerned with the regular paperfolding
sequence, although there are extensions to all paperfolding sequences
discussed in Section~\ref{further}.

\section{The sequences}

In this section we introduce the sequences we will study:
\begin{itemize}
\item the regular paperfolding sequence $(p_n)$ defined above.

\item $(q_n)_{n \geq 1}$, a version of the paperfolding sequence defined over
the alphabet $\{0,1\}$, and satisfying $p_n = (-1)^{q_n}$.

\item the first difference sequence $(d_n)_{n \geq 1}$ of the sequence
$(q_n)_{n \geq 1}$,
defined by $d_1 = 1$ and $d_n = q_n - q_{n-1}$ for $n \geq 2$.

\item the absolute value of the sequence $(d_n)_{n \geq 1}$, which
we denote by $(d'_n)_{n \geq 1}$.

\item the ending positions of the runs of $d'_n$, which we denote
by $e'_n$.  It is useful to set $e'_0 = 0$.

\item the starting positions of the runs of $d'_n$, which
we denote by $s'_n$.

\item the run lengths of the sequence $(d'_n)$, which we denote
by $(b_n)_{n \geq 1}$.  We will see later that $b_n = a_n$, which
will give a connection between the paperfolding sequence and
Cloitre's sequence.

\item the successive
ending positions $e_n$ of the successive runs appearing in the sequence
$(b_n)$, and

\item the starting positions $s_n$ of these runs.

\item the run lengths $r_n$ of the runs appearing in $(b_n)_{n \geq 1}$.

\item the sums of the elements appearing in the $n$'th run
of $(b_n)$, which we denote by $\sigma_n$.  

\item the number $g_n$ of $1$'s appearing in the prefix $b_1 \cdots b_n$.

\item the sequence $h_n = 3g_n - 2n$.
\end{itemize}
The first few values of each of these sequences are given in Table~\ref{tab1}.
\begin{table}[H]
\begin{center}
\begin{tabular}{c|ccccccccccccccccccc}
$n$ & 1 & 2 &  3 & 4 & 5 & 6 & 7 & 8 & 9 & 10 & 11 & 12 & 13 & 14 & 15 & OEIS \\
\hline
$a_n$ & 1&1&2&1&1&1&1&2&1&1&2&1&1&2&2& \seqnum{A157196} \\
$p_n$ & 1& 1&$-1$& 1& 1&$-1$&$-1$& 1& 1& 1&$-1$&$-1$& 1&$-1$&$-1$& \seqnum{A034947}\\
$q_n$ & 0&0&1&0&0&1&1&0&0&0&1&1&0&1&1& \seqnum{A014707}\\ 
$d_n$ & 1& 0& 1&$-1$& 0& 1& 0&$-1$& 0& 0& 1& 0&$-1$& 1& 0&\\
$d'_n$ & 1&0&1&1&0&1&0&1&0&0&1&0&1&1&0& \seqnum{A379728}\\
$e'_n$ & 1& 2& 4& 5& 6& 7& 8&10&11&12&14&15&16&18&20&\\ 
$s'_n$ & 1& 2& 3& 5& 6& 7& 8& 9&11&12&13&15&16&17&19&\\
$b_n$ & 1&1&2&1&1&1&1&2&1&1&2&1&1&2&2& \seqnum{A157196} \\
$e_n$ & 2& 3& 7& 8&10&11&13&15&17&18&22&23&25&27&31& \seqnum{A379729} \\ 
$s_n$ & 1& 3& 4& 8& 9&11&12&14&16&18&19&23&24&26&28&\\
$r_n$ & 2&1&4&1&2&1&2&2&2&1&4&1&2&2&4& \seqnum{A379704}\\
$\sigma_n$ & 2&2&4&2&2&2&2&4&2&2&4&2&2&4&4&\\
$g_n$ & 1& 2& 2& 3& 4& 5& 6& 6& 7& 8& 8& 9&10&10&10&\\ 
$h_n$ & 1&2&0&1&2&3&4&2&3&4&2&3&4&2&0&
\end{tabular}
\end{center}
\caption{The sequences we study in this paper.}
\label{tab1}
\end{table}

\section{The connection between paperfolding and Cloitre's sequence}

We use the free software {\tt Walnut}
to help with proofs \cite{Mousavi:2016,Shallit:2023}. Given a first-order
formula about integers involving addition, logical operations, finite
automata, and the universal and existential quantifiers, this software
can provide rigorous proofs or disproofs of assertions.

The idea is that each of the
sequences we study is, {\it mirabile dictu},
computable using a finite automaton.  There are two different
ways to be so computable, depending on the type of sequence.
If the sequence $z_n$ takes integer
values, then we can use a {\it synchronized automaton} 
\cite{Shallit:2021}.
Such an automaton accepts two inputs in parallel, $n$ and $z$, and accepts
if and only if $z = z_n$.  If the sequence takes only finitely many integer
values, then we can also use a {\it deterministic finite automaton with
output} (DFAO), where an output is associated with the last state reached
\cite{Allouche&Shallit:2003}.  In all of the automata in this paper,
the inputs are expressed in base $2$,
{\it starting with the least
significant bit}.

We start with the regular paperfolding sequence $q_n$ over the alphabet
$\lbrace 0, 1 \rbrace$.  We index it, like all of our sequences, starting
with $n =1$.  However, it is useful to define $q_0 = 0$.
With this definition, the sequence obeys the 
recurrence $q_{2n} = q_n$, $q_{4n+1} = 0$, $q_{4n+3} = 1$ for $n \geq 0$,
as is well known \cite{Dekking&MendesFrance&vanderPoorten:1982}.
Also in \cite{Dekking&MendesFrance&vanderPoorten:1982} one
can find the $4$-state 
DFAO for it, illustrated in Figure~\ref{fig1}, and called
{\tt Q}.
We can check that it obeys the definition as follows:
\begin{verbatim}
eval q_test "?lsd_2 An (Q[2*n]=Q[n] & Q[4*n+1]=@0 & Q[4*n+3]=@1)":
\end{verbatim}
and {\tt Walnut} returns {\tt TRUE}.
\begin{figure}[htb]
\begin{center}
\vspace{-.4in}
\includegraphics[width=3.5in]{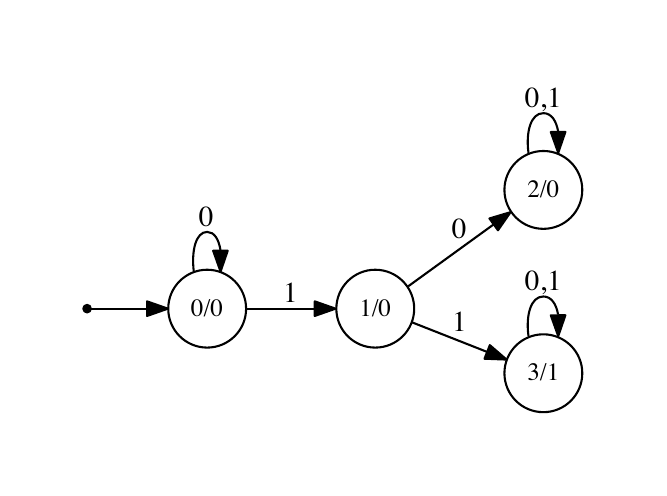}
\vspace{-.4in}
\end{center}
\caption{The lsd-automaton for $(q_n)$.}
\label{fig1}
\end{figure}

Now that we are giving {\tt Walnut} commands, it is time to briefly
explain {\tt Walnut}'s syntax.  
The {\tt Walnut} commands {\tt def} and {\tt reg}
create automata that can be used in later calculations.
The command {\tt eval} returns a result
{\tt TRUE} or {\tt FALSE}.  The jargon {\tt ?lsd\_2} indicates
that integers are represented in base $2$ in least-significant-digit
first format.  The symbols {\tt \&}, {\tt |}, {\tt =>}, {\tt <=>}
represent logical {\tt AND}, {\tt OR}, implication, and {\tt IFF},
respectively.  The symbol {\tt A} represents the universal quantifier
$\forall$ and {\tt E} represents the existential quantifier $\exists$.
The {\tt @} sign precedes a numerical constant that can be the output
of an automaton.  The {\tt combine} command allows one to take a number
of finite automata and turn them into a DFAO with different prescribed
outputs.  The {\tt minimize} command allows us to find an equivalent
automaton with the smallest possible number of states.

Now let us create the automaton for $d_n$, the first difference of
the paperfolding sequence.  It seems natural to define this using
{\tt Q} and subtraction, but since the default domain for numbers
is $\Enn$, the non-negative integers, we must use a trick and substitute
some suitable value, such as $2$, for $-1$, and then replace it later.
We first create a synchronized automaton for this, and then convert it
into a DFAO.
\begin{verbatim}
def synchd "?lsd_2 (n=0 & z=0) | (n=1 & z=1) | 
   (z=1 & Q[n]=@1 & Q[n-1]=@0) | (n>=2 & z=0 & Q[n]=Q[n-1]) |
   (z=2 & Q[n]=@0 & Q[n-1]=@1)":
def d0 "?lsd_2 $synchd(n,0)":
def d1 "?lsd_2 $synchd(n,1)":
def d2 "?lsd_2 $synchd(n,2)":
combine D d0=0 d1=1 d2=-1:
\end{verbatim}
This $12$-state automaton, called {\tt D}, is displayed in Figure~\ref{fig2}.
\begin{figure}[htb]
\begin{center}
\vspace{-.2in}
\includegraphics[width=5in]{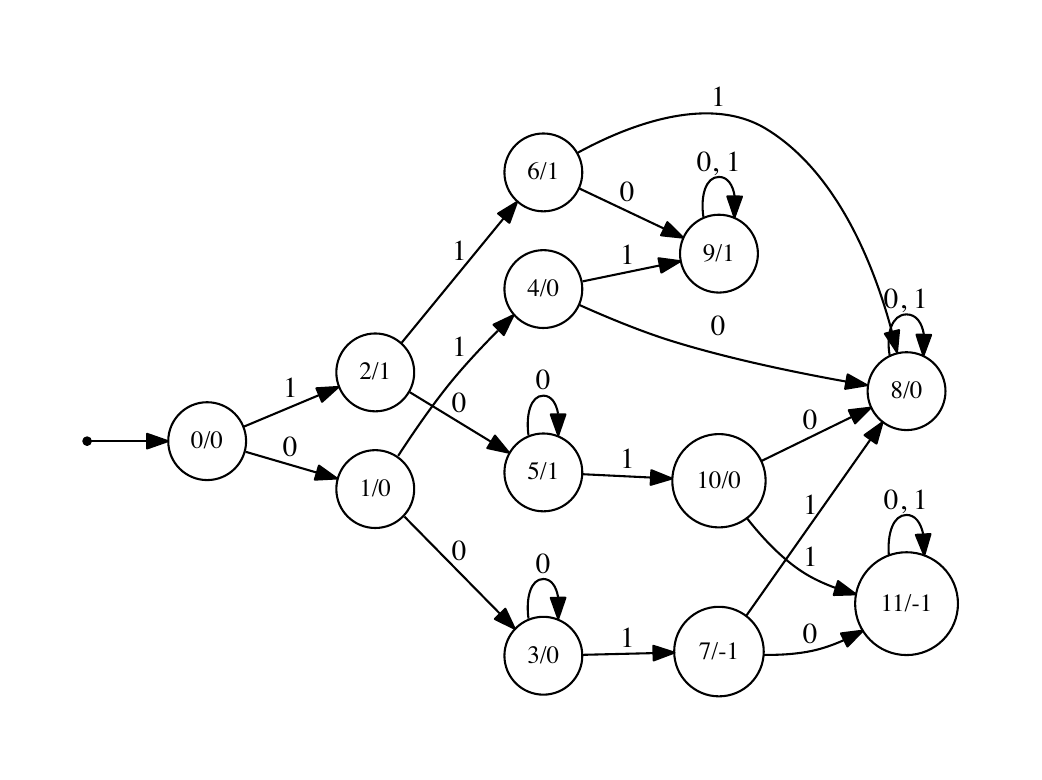}
\vspace{-.4in}
\end{center}
\caption{The lsd-automaton {\tt D} for $(d_n)$.}
\label{fig2}
\end{figure}

Next we create the DFAO {\tt DP}  for $(d'_n)$.
\begin{verbatim}
def dp12 "?lsd_2 $synchd(n,1) | $synchd(n,2)":
combine DPP dp0=0 dp12=1:
minimize DP DPP:
\end{verbatim}
This $9$-state automaton, called {\tt DP}, is displayed in Figure~\ref{fig3}.
\begin{figure}[htb]
\begin{center}
\includegraphics[width=5in]{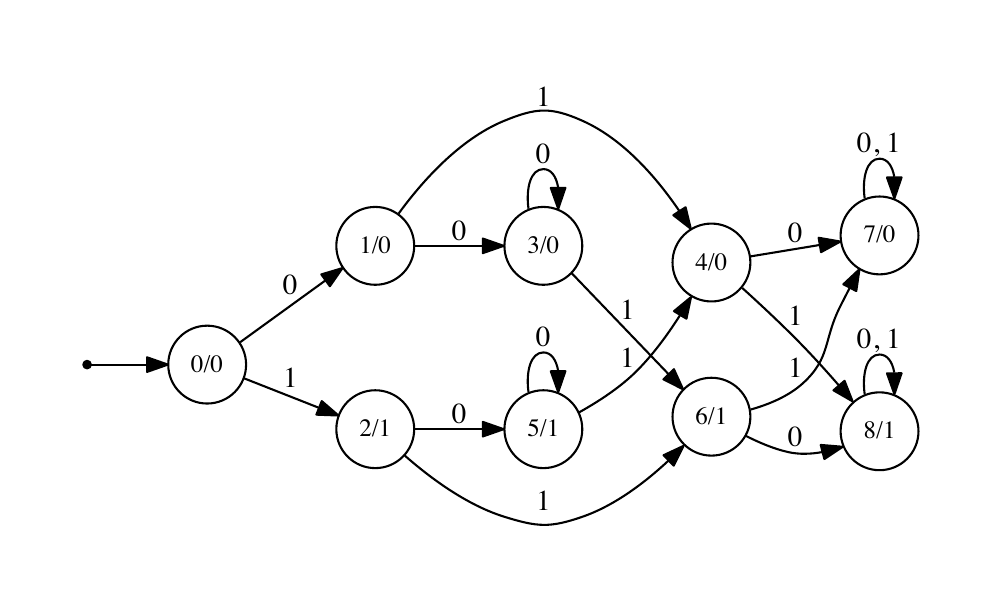}
\vspace{-.4in}
\end{center}
\caption{The lsd-automaton {\tt DP} for $(d'_n)$.}
\label{fig3}
\end{figure}

Next we turn to creating a synchronized
automaton {\tt ep} for $(e'_n)$, the ending positions
of the runs of $(d'_n)$. In general, there is no
guarantee that an automaton for these ending positions 
for an automatic sequence
will exist, so it seems difficult to get an automaton for $(e'_n)$
through some sort of deterministic procedure.
Instead, we use a tactic discussed
before in \cite{Shallit:2023}: we {\it guess\/} the automaton from empirical data,
and then, once guessed, we can verify its correctness {\it rigorously}
through {\tt Walnut}.  The guessed automaton, which has 19 states,
is displayed in
Figure~\ref{fig4}.
\begin{figure}[htb]
\begin{center}
\includegraphics[width=6.5in]{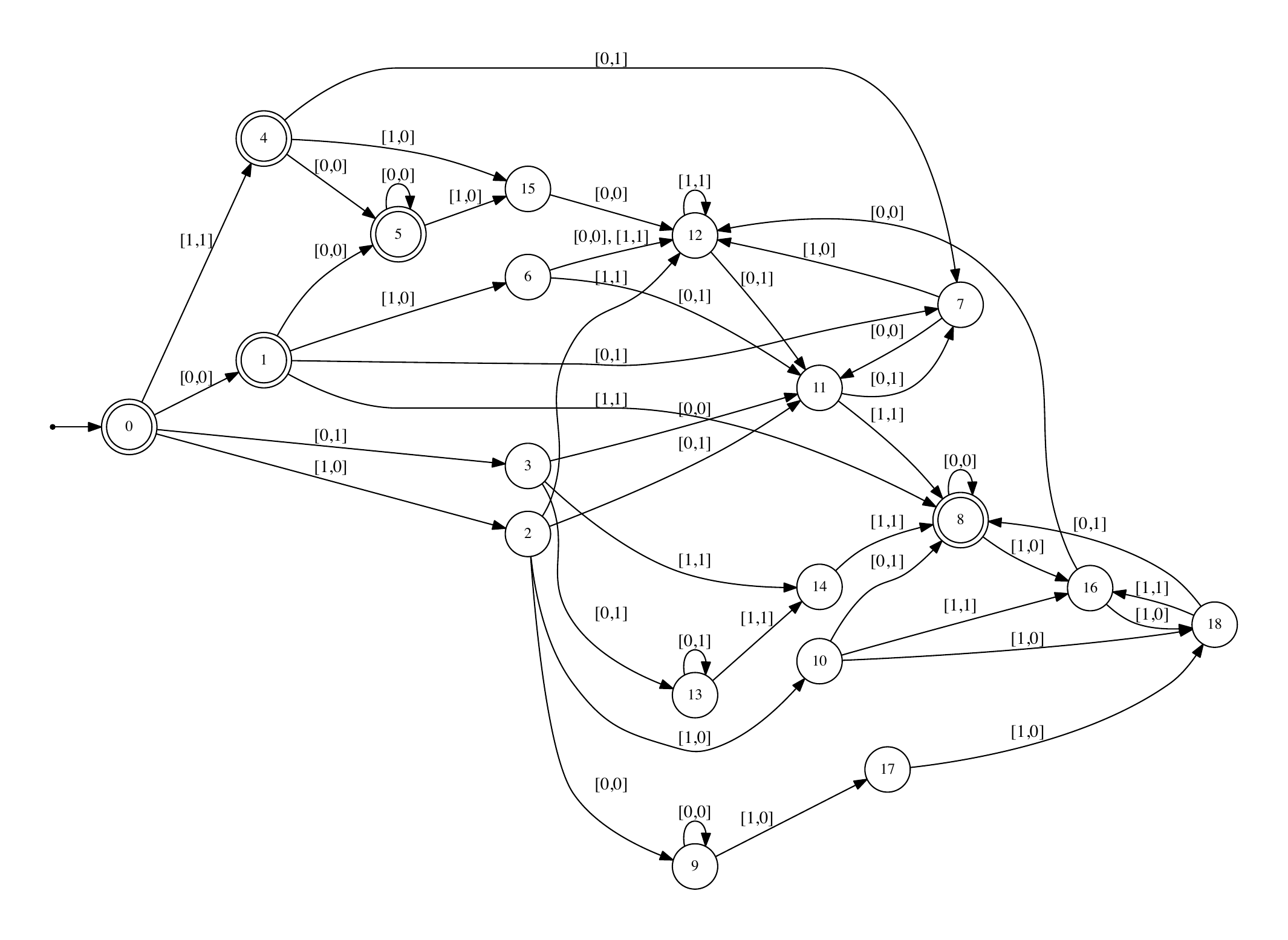}
\vspace{-.4in}
\end{center}
\caption{The lsd-automaton {\tt ep} for $(e'_n)$.}
\label{fig4}
\end{figure}

To verify its correctness, we use {\tt Walnut} to
check that the automaton {\tt ep} really
computes a strictly increasing
function $e''_n$, that the symbol of $d'$ at position $e''_n$ differs
from that at position $e''_n + 1$, that all the symbols
from position $e''_{n-1} +1$ to $e''_n - 1$ are the same, and
that $e''_1 = 1$.  
We do this as follows:
\begin{verbatim}
eval ep1 "?lsd_2 An Ex $ep(n,x)":
# there is a value for each non-negative integer n
eval ep2 "?lsd_2 ~En,x,y x!=y & $ep(n,x) & $ep(n,y)":
# there is no n for which ep takes two distinct values
eval ep3 "?lsd_2 An,x,y ($ep(n,x) & $ep(n+1,y)) => x<y":
# the function ep computes is strictly increasing
eval ep4 "?lsd_2 An,x $ep(n,x) => DP[x]!=DP[x+1]":
# the symbol of d'_n at position e'_n differs from that at position e'_n + 1
eval ep5 "?lsd_2 An,x,y,t ($ep(n-1,x) & $ep(n,y) & t>=x+2 & t<y) => 
  DP[t]=DP[x+1]":
# all the symbols from e'_{n-1}+1 to e'_n - 1 are the same
eval ep6 "?lsd_2 $ep(1,1)":
# e'_1 = 1
\end{verbatim}
All of these return {\tt TRUE}.
Now an easy induction proves that $e'_n = e''_n$.
The same trick appeared in a recent paper \cite{Shallit:2024}.
%Since the ideas and the {\tt Walnut} code for this have recently
%appeared in a related paper \cite{Shallit:2024}, we omit the details.

Once we have the automaton {\tt ep} for $e'_n$, we can trivially
obtain the $20$-state
automaton {\tt sp} for $s'_n$, the starting positions of runs of
$d'_n$, as follows:
\begin{verbatim}
def sp "?lsd_2 (n=0 & z=0) | $ep(n-1,z-1)":
\end{verbatim}
And once we have the automaton for $s'_n$ and $e'_n$, we can
obtain a $32$-state synchronized automaton {\tt b} for $b_n$:
\begin{verbatim}
def b "?lsd_2 Ex,y $sp(n,x) & $ep(n,y) & z=(y-x)+1":
\end{verbatim}
We can now check that $b_n$ is either $1$ or $2$.
\begin{verbatim}
eval b12 "?lsd_2 An (n>=1) => ($b(n,1) | $b(n,2))":
\end{verbatim}
which evaluates to {\tt TRUE}.
Finally, we can get a DFAO {\tt B} for $b_n$ as follows:
\begin{verbatim}
def b1 "?lsd_2 $b(n,1)":
def b2 "?lsd_2 $b(n,2)":
combine B b1=1 b2=2:
\end{verbatim}
This $22$-state automaton is displayed in Figure~\ref{fig5}.
\begin{figure}[htb]
\begin{center}
\includegraphics[width=6in]{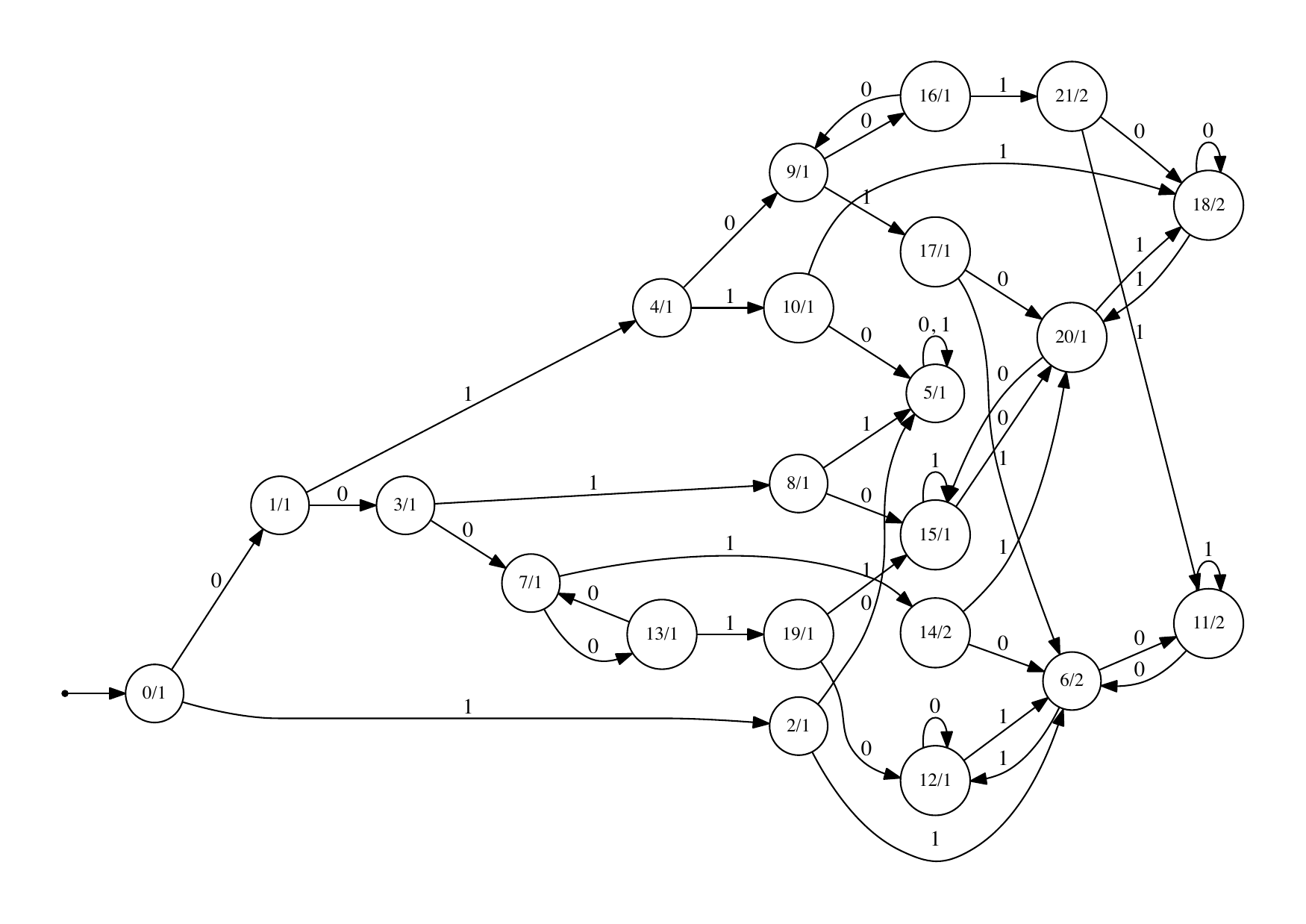}
\end{center}
\caption{The lsd-automaton {\tt B} for $(b_n)$.}
\label{fig5}
\end{figure}
Notice that from the way we have defined the sequences, it follows
immediately that that $e'_n = \sum_{1 \leq i \leq n} b_i$.

It's now time to verify that the sum of the runs appearing in
$b_n$ have the desired property.  To do that we first need {\tt e}, a
synchronized automaton for $e_n$,
the ending position of runs in $b_n$.   We adopt the same strategy
as before:  guess the automaton and then verify it is correct.
\begin{figure}[htb]
\begin{center}
\includegraphics[width=6.5in]{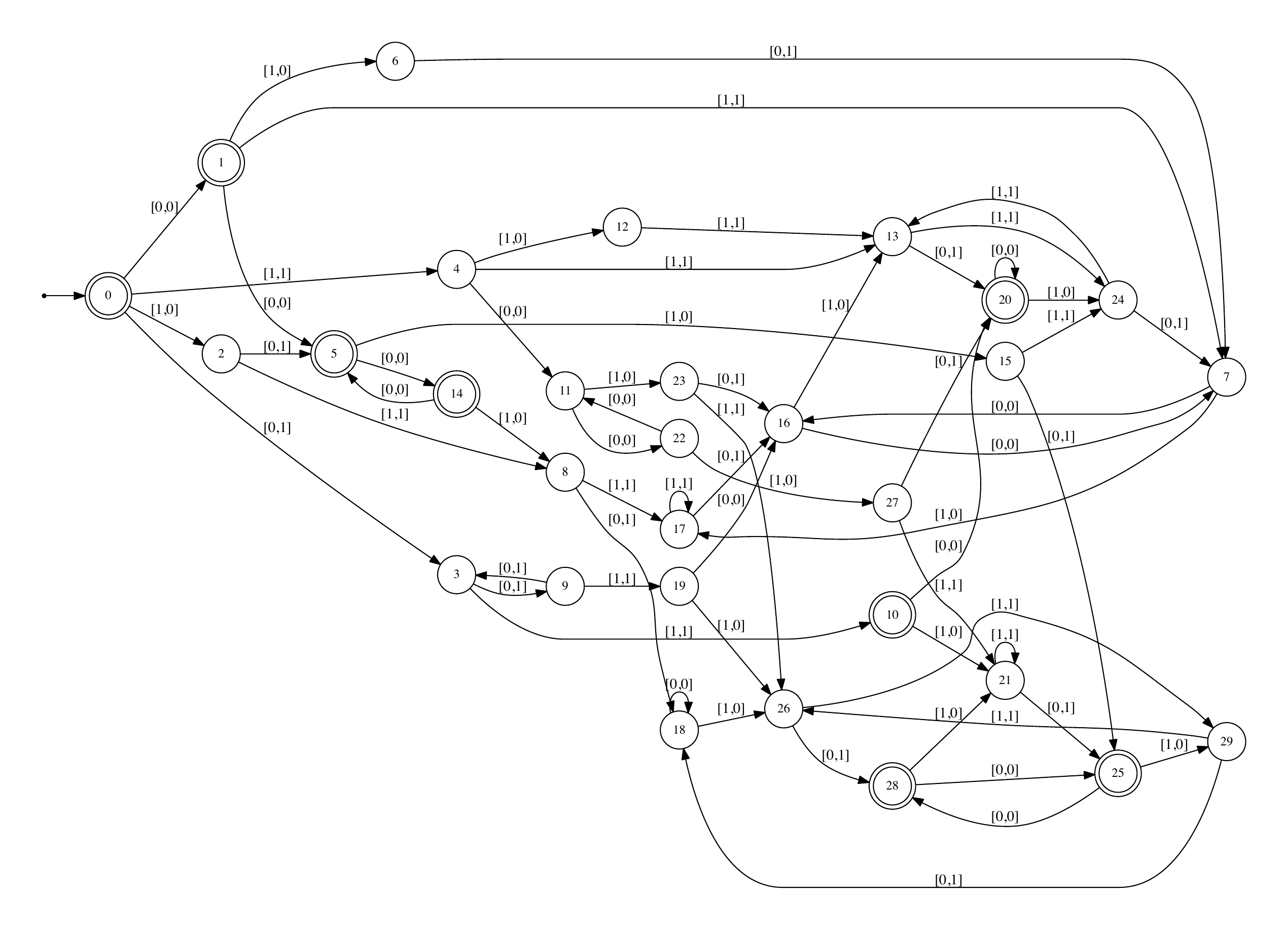}
\end{center}
\caption{The synchronized automaton for $(e_n)$.}
\label{fig7}
\end{figure}
It has $30$ states.
We can verify the correctness of the automaton {\tt e} in exactly
the same way that we did for {\tt ep} above.   We omit the details.

We can now compute the starting positions of runs in much the same
way as we did for {\tt sp} above.  The following command creates
a $32$-state automaton {\tt s}.
\begin{verbatim}
def s "?lsd_2 (n=0 & z=0) | $e(n-1,z-1)":
\end{verbatim}

We can now compute the lengths of the runs $r_n$ in exactly the same
way we did for $(b_n)$ above.
\begin{verbatim}
def r "?lsd_2 Ex,y $s(n,x) & $e(n,y) & z=(y-x)+1":
\end{verbatim}
This gives us a $32$-state synchronized automaton {\tt r}.
We observe that the only runs in $(b_n)$ are of length $1,2,$ or $4$.
\begin{verbatim}
def r_check "?lsd_2 An,x $r(n,x) => (x=1 | x=2 | x=4)":
\end{verbatim}
And {\tt Walnut} returns {\tt TRUE}.
Notice that, from the way we have defined the sequences, it follows
immediately that $e_n = \sum_{1 \leq i \leq n} r_i$.

Next, we can compute a synchronized automaton {\tt sigma} for the sum of the
values of each run, using the fact that the runs of $1$'s and $2$'s
alternate.
\begin{verbatim}
reg even lsd_2 "()|0(0|1)*":
reg odd lsd_2 "1(0|1)*":
def sigma "?lsd_2 Ex,y ($odd(n) & $s(n,x) & $e(n,y) & z=(y-x)+1) | 
   ($even(n) & $s(n,x) & $e(n,y) & z=2*((y-x)+1))":
\end{verbatim}
This automaton has $31$ states.

Finally, let us check that $\sigma_n = 2 b_n$:
\begin{verbatim}
eval b_property "?lsd_2 An,x,y ($b(n,x) & $sigma(n,y)) => y=2*x":
\end{verbatim}
And {\tt Walnut} returns {\tt TRUE}.

At this point we now know that indeed $a_n = b_n$; that is, the
sequence $b_n$ we constructed starting from the regular paperfolding
sequence is indeed Cloitre's sequence $a_n$.

\section{The proof of Cloitre's conjecture}

We can now proceed to the proof of our main result, which resolves
Cloitre's conjecture:
\begin{theorem}
Let $g_n$ be the number of $1$'s in the sequence $a_1 a_2 \cdots a_n$.
Then 
\begin{equation}
0 \leq 3g_n -2n \leq 4
\label{star}
\end{equation}
for all $n$, and hence
$\lim_{n \rightarrow \infty} g_n/n = 2/3$.
\label{thm1}
\end{theorem}

\begin{proof}
We adopt a similar strategy as before:  we guess the synchronized automaton
for $g_n$ and then verify it is correct.  It has $16$ states
and is displayed in
Figure~\ref{fig6}.
\begin{figure}[htb]
\begin{center}
\includegraphics[width=6.5in]{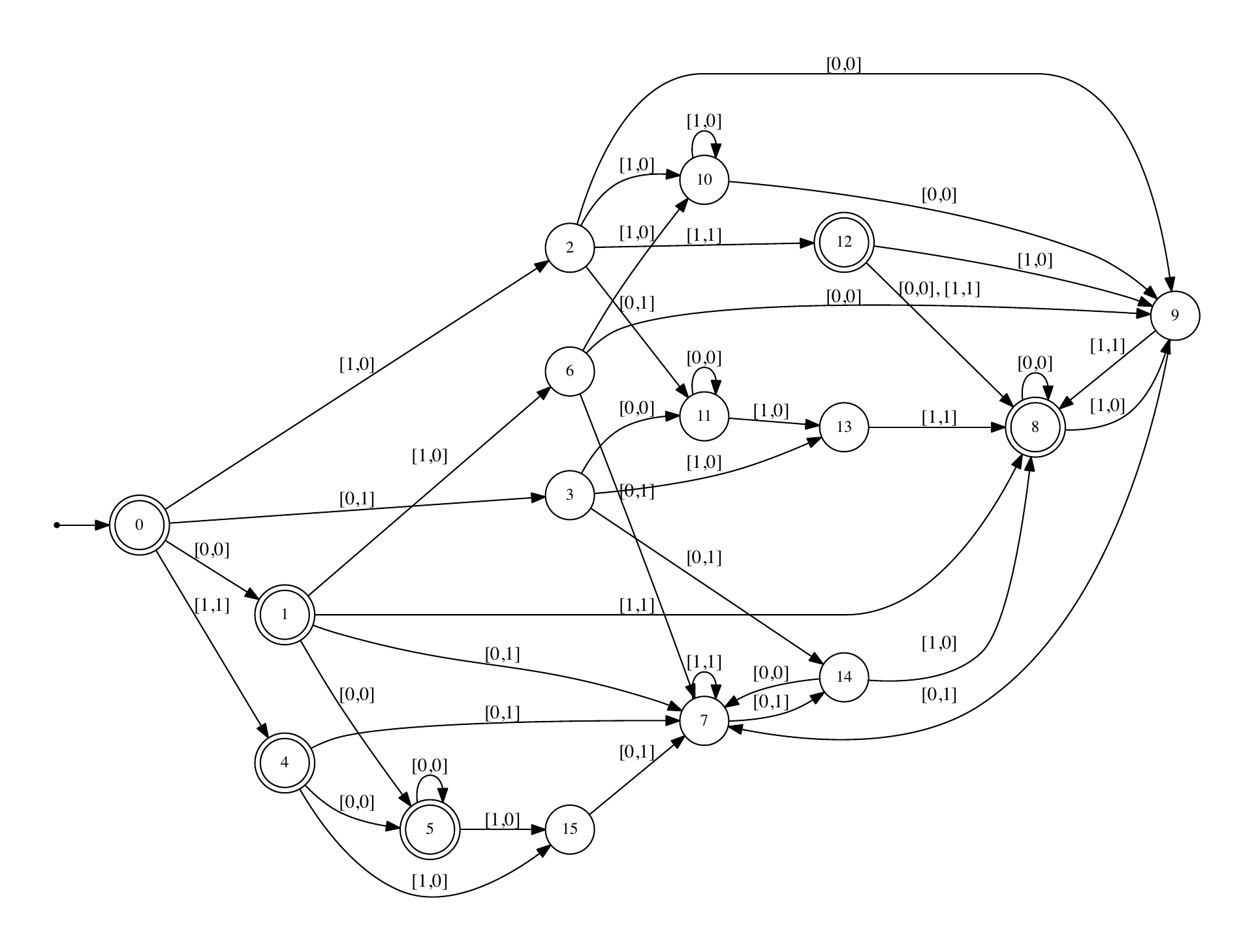}
\end{center} 
\caption{The synchronized automaton for $(g_n)$.}
\label{fig6}
\end{figure} 
To verify it is correct, we need to check that
$g_n = g_{n-1}+1$ if $b_n = 1$, and
$g_n = g_{n-1}$ if $b_n = 2$. We can do that as follows:
\begin{verbatim}
eval g_correctness "?lsd_2 $g(0,0) & 
   An (n>=1) => (Ex ($g(n,x) & $g(n-1,x)) <=> $b(n,2))":
\end{verbatim}

Finally, we can verify the inequality \eqref{star} as follows:
\begin{verbatim}
eval inequality_check "?lsd_2 An,x $g(n,x) => (3*x>=2*n & 3*x<=2*n+4)":
\end{verbatim}
And {\tt Walnut} returns {\tt TRUE}. This completes the proof.
\end{proof}

\section{Another result}

Benoit Cloitre (personal communication) conjectured that if we define
$$w_n = \min \{t \geq 1 \, : \, e'_t \geq n \},$$
then $(w_n)_{n \geq 1}$ is sequence \seqnum{A091960} in the OEIS.

Using his definition,
we can define a synchronized automaton for $w_n$ as follows:
\begin{verbatim}
def w "?lsd_2 Ex $ep(t,x) & x>=n & Au,y ($ep(u,y) & y>=n) => u>=t":
\end{verbatim}
Once we have it, we can verify that indeed $w_n$ satisfies
the construction rules given in the OEIS entry for \seqnum{A091960},
namely, that $w_1 = 1$ and $w_{2n} = w_{2n-1} + (w_n \bmod 2)$,
and $w_{2n+1} = w_{2n} + 1$ for $n \geq 1$.
\begin{verbatim}
def mod2 "?lsd_2 n=z+2*(n/2)":
eval w_check "?lsd_2 $w(1,1) & 
   (An,x,y,z,t (n>=1 & $w(2*n,x) & $w(2*n-1,y) & $w(n,z)
   & $mod2(z,t)) => x=y+t) &
   (An,x,y (n>=1 & $w(2*n+1,x) & $w(2*n,y)) => x=y+1)":
\end{verbatim}
And {\tt Walnut} returns {\tt TRUE}.

\section{Going further}
\label{further}

As in \cite{Goc&Mousavi&Schaeffer&Shallit:2015,Shallit:2023}, there is a single finite automaton that can compute
the values of any finite paperfolding sequence.  The input is the sequence
of unfolding instructions $f$ and an integer $n$, and the result is the
$n$'th term of the associated paperfolding sequence.

Using this automaton, we can find the analogues of the sequences
of Table~\ref{tab1} more generally; there is a {\it single\/} finite automaton
that computes the sequence for {\it all\/}
finite sequences of folding instructions
simultaneously.
Once again, the method it involves guessing some of the automata and using
{\tt Walnut} to verify their correctness; then computing others from these.

Although there is no obvious analogue of the self-generating property for an
arbitrary paperfolding sequence, all of the other results hold for an
arbitrary paperfolding sequence, including Theorem~\ref{thm1}.

\section{Acknowledgments}

I thank Benoit Cloitre for sharing his conjecture about $e'_n$, $w_n$,
and \seqnum{A091960}.  Thanks once again to the OEIS for providing a great
resource.


\begin{thebibliography}{99}

\bibitem{Allouche&Shallit:2003}
J.-P. Allouche and J. Shallit.
\newblock {\em Automatic Sequences: Theory, Applications, Generalizations}.
\newblock Cambridge University Press, 2003.

\bibitem{Davis&Knuth:1970}
C.~Davis and D.~E. Knuth.
\newblock Number representations and dragon curves--{I, II}.
\newblock {\em J. Recreational Math.} {\bf 3} (1970), 66--81, 133--149.

\bibitem{Dekking:1997}
F. M. Dekking.
\newblock What is the long range order in the Kolakoski sequence?
\newblock In R. V. Moody,
ed., {\it The Mathematics of Long-Range Aperiodic Order}.
\newblock {\it NATO ASI Series C Mathematical and Physical Sciences---Advanced Study Institute} {\bf 489} (1997), pp.~115--125.

\bibitem{Dekking&MendesFrance&vanderPoorten:1982}
F.~M. Dekking, M.~{Mend\`es}~France, and A.~J. {van der Poorten}.
\newblock Folds!
\newblock {\em Math. Intelligencer} {\bf 4} (1982), 130--138, 173--181,
  190--195.
\newblock Erratum, {\bf 5} (1983), 5.

\bibitem{Goc&Mousavi&Schaeffer&Shallit:2015}
D. Go\v{c}, H. Mousavi, L. Schaeffer, and J. Shallit.
\newblock A new approach to the paperfolding sequences.
\newblock In A. Beckmann et al., eds., {\it CiE 2015},
Lecture Notes in Comput. Sci., Vol.~9136, Springer, 2015, pp.~34--43.

\bibitem{Golomb:1966}
S. W. Golomb.
\newblock Run-length encodings.
\newblock {\it IEEE Trans. Info. Theory} {\bf IT-12} (1966), 399--401.

\bibitem{Hopcroft&Ullman:1979}
J.~E. Hopcroft and J.~D. Ullman.
\newblock {\em Introduction to Automata Theory, Languages, and Computation}.
\newblock Addison-Wesley, 1979.

\bibitem{Kolakoski:1965}
W. Kolakoski.
\newblock Problem 5304.
\newblock {\it Amer. Math. Monthly} {\bf 72} (1965), 674.
Solution in {\it Amer. Math. Monthly} {\bf 73} (1966), 681--682.

\bibitem{Mousavi:2016}
H.~Mousavi.
\newblock Automatic theorem proving in {{\tt Walnut}}.
\newblock Arxiv preprint arXiv:1603.06017 [cs.FL], available at
  \url{http://arxiv.org/abs/1603.06017}, 2016.

\bibitem{Oldenburger:1939}
R. Oldenburger.
\newblock Exponent trajectories in symbolic dynamics.
\newblock {\it Trans. Amer. Math. Soc.} {\bf 46} (1939), 453--466.

\bibitem{Shallit:2021}
J. Shallit.
\newblock Synchronized sequences.
\newblock In T. Lecroq and S. Puzynina, eds., {\it WORDS 2021},
Lecture Notes in Comp. Sci., Vol.~12847, Springer, 2021, pp.~1--19.

\bibitem{Shallit:2023}
J.~Shallit.
\newblock {\em The Logical Approach To Automatic Sequences: Exploring
  Combinatorics on Words with {\tt Walnut}}, Vol. 482 of {\em London Math. Soc.
  Lecture Note Series}.
\newblock Cambridge University Press, 2023.

\bibitem{Shallit:2024}
J. Shallit.
\newblock Runs in paperfolding sequences.
\newblock ArXiv preprint arXiv:2412.17930 [math.CO], December 23 2024. 
\newblock Available at \url{https://arxiv.org/abs/2412.17930}.

\bibitem{oeis}
N. J. A. Sloane et al.
\newblock The On-Line Encyclopedia of Integer Sequences.
\newblock Available at \url{https://oeis.org}.


\end{thebibliography}
\end{document}